\documentclass[12pt, reqno]{amsart}

\usepackage{verbatim} 

\usepackage{amsmath, amssymb, amsfonts, amsbsy, amsthm, latexsym}
\usepackage{mathabx}
\usepackage[normalem]{ulem}
\usepackage{algorithm}
\usepackage{algorithmic}

\usepackage{enumitem}

\usepackage[dvipsnames]{xcolor}
\usepackage{graphicx}
\usepackage{caption}
\usepackage{subcaption}

\newtheorem{theorem}{Theorem}[section]
\newtheorem{proposition}[theorem]{Proposition}
\newtheorem{lemma}[theorem]{Lemma}
\newtheorem{corollary}[theorem]{Corollary}

\DeclareMathOperator*{\argmax}{arg\,max}

\theoremstyle{definition}
\newtheorem{definition}[theorem]{Definition}
\newtheorem{example}[theorem]{Example}
\newtheorem{construction}[theorem]{Construction}
\newtheorem{remark}[theorem]{Remark}

\numberwithin{equation}{section}

\usepackage[colorlinks=true, linkcolor=blue, citecolor=blue, urlcolor=blue]{hyperref}

\def\C{{\mathbb{C}}}

\def\G{{\mathcal{G}}}

\def\N{{\mathbb{N}}}

\def\R{{\mathbb{R}}}
\def\Z{{\mathbb{Z}}}

\def\supp{{\mathrm{supp}}}
\usepackage{bbm}

\definecolor{bmcolor}{rgb}{0.9, 0.3, 0}

\definecolor{ikcolor}{rgb}{0, .4, .4}

\title{Gabor duals with minimal L1-norm}

\author{Rocío Díaz Martín}
\address{\textsc{Rocío Díaz Martín}\\ Department of Mathematics, Florida State University, Tallahassee, FL, USA}
\email{rdiazmartin@fsu.edu}

\author{Mark Lammers}
\address{\textsc{Mark Lammers}\\ Department of Mathematics and Statistics, University of North Carolina Wilmington, Wilmington, NC, USA}
\email{lammersm@uncw.edu}

\author{Brendan Miller}
\address{\textsc{Brendan Miller}\\ Department of Mathematics, Vanderbilt University, Nashville, TN, USA}
\email{brendan.miller@vanderbilt.edu}

\author{Alexander M. Powell}
\address{\textsc{Alexander M. Powell}\\ Department of Mathematics, Vanderbilt University, Nashville, TN, USA}
\email{alexander.m.powell@vanderbilt.edu}

\begin{document}

\begin{abstract}
    It is well-known that for a Gabor frame in $L^2(\mathbb{R})$, the canonical dual window is the dual window with the minimal $L^2$-norm. In this paper, we address the problem of finding dual windows for a Gabor frame that minimizes the $L^1$-norm. Since it has been shown that minimizing the $L^1$-norm often coincides with minimizing the $L^0$-norm, we also investigate the support of the dual Gabor window that achieves the  minimal $L^1$-norm.
\end{abstract}


\maketitle

\section{Introduction and Preliminaries}\label{sec: Intro}

Gabor frames play a fundamental role in time-frequency analysis, providing flexible and stable representations of signals through redundant expansions. While dual Gabor frames are not unique, the canonical dual is distinguished by minimizing the $L^2$-energy of the expansion coefficients. In contrast, minimizing the $L^1$-norm tends to promote sparsity and localization, which are desirable in many applications. In this work, we investigate the problem of identifying dual Gabor windows with minimal $L^1$-norm, with particular emphasis on understanding their structure and support.

\subsection{Frames}
Frames are one of the central tools in signal processing analysis \cite{duffin1952class}. In particular, we say that a system of vectors $\{g_n\}_{n\in\mathbb{N}} \subseteq L^2(\mathbb{R})$ is a frame for $L^2(\mathbb{R})$ if there exist positive constants $A, B$ such that 
\begin{equation}\label{eq:frame_cond}
    A\|f\|_2^2 \leq \sum_{n\in\mathbb{N}} |\langle f, g_n \rangle|^2 \leq B\|f\|_2^2, \qquad \forall f \in L^2(\mathbb{R}),
\end{equation}
where $\langle f,g\rangle:=\int_\R f(x) \overline{g(x)}dx$. If only the upper bound in \eqref{eq:frame_cond} holds, we say that $\{g_n\}_{n\in\N}$ is a Bessel sequence. 
When both upper and lower bound conditions hold as in \eqref{eq:frame_cond}, they imply that the analysis operator $f \mapsto (\langle f, g_n \rangle)_{n\in\mathbb{N}}$ is continuous from $L^2(\mathbb{R})$ to $\ell^2$ and invertible on its range \cite[Thm 2.29 $\&$ Thm 8.29]{heil2011basis}, allowing for stable reconstructions of the form
\begin{equation}\label{eq:dual_frame}
    f = \sum_{n\in\mathbb{N}} \langle f, g_n \rangle \,  h_n, \qquad \forall f \in L^2(\mathbb{R}),
\end{equation}
for some system of vectors $\{h_n\}_{n\in\mathbb{N}} \subset L^2(\mathbb{R})$ called a \textit{dual frame}. In general, such a system $\{h_n\}_{n\in\mathbb{N}}$ satisfying \eqref{eq:dual_frame} is not unique. However, a canonical way of finding such a system of vectors relies on considering the frame operator $S: L^2(\mathbb{R}) \to L^2(\mathbb{R})$ defined as the composition of the analysis operator and its adjoint operator (the synthesis operator):
$$
    S f := \sum_{n\in\mathbb{N}} \langle f, g_n \rangle \, g_n, \qquad \forall f \in L^2(\mathbb{R}).
$$
Under the frame condition \eqref{eq:frame_cond},  $S$ is self-adjoint and invertible with a bounded inverse $S^{-1}: L^2(\mathbb{R}) \to L^2(\mathbb{R})$, giving that
$$
    f = SS^{-1}f = \sum_{n\in\mathbb{N}} \langle f, S^{-1}g_n \rangle \, g_n \quad \text{and} \quad f = S^{-1}Sf = \sum_{n\in\mathbb{N}} \langle f, g_n \rangle \, S^{-1}g_n, \qquad \forall f \in L^2(\mathbb{R})
$$
\cite[Lemma 5.1.5 $\&$ Thm 5.1.6]{christensen2003introduction}.
Thus, $\{S^{-1}g_n\}_{n\in\mathbb{N}}$ is a dual frame for $\{g_n\}_{n\in\mathbb{N}}$, called the \textit{canonical dual frame}. It is well known that the canonical dual frame satisfies the following inequality for every $f \in L^2(\mathbb{R})$:
$$
    \sum_{n\in\mathbb{N}} |\langle f, S^{-1}g_n \rangle|^2 \leq \sum_{n\in\mathbb{N}} |c_n|^2, \qquad \forall (c_n) \in \ell^2 \text{ such that } f = \sum_{n\in\mathbb{N}} c_n g_n,
$$
that is, canonical dual frame coefficients $(\langle f, S^{-1}g_n \rangle)_{n\in \N}$ 
have minimum energy among all $\ell^2$-coefficient sequences synthesizing $f$ from the frame $\{g_n\}_{n\in\mathbb N}$ \cite[Lemma 5.4.2]{christensen2003introduction}.
This means that when representing a signal using a dual frame $\{h_n\}_{n\in \N}$, the coefficients obtained from the canonical dual frame $\{S^{-1}g_n\}_{n\in \N}$ will have the smallest possible total energy in the $L^2$ sense. For a more detailed treatment of frames and bases, we refer the reader to the excellent textbooks on the subject, including \cite{christensen2003introduction,grochenig2013foundations,heil2011basis}.

\subsection{Gabor Frames} In this paper, we focus on Gabor frames, which are obtained by translating a single
window in time and frequency.
Given $ g \in L^2(\mathbb{R}) $ and parameters $\alpha, \beta > 0$, the associated \textit{Gabor system} \cite{gabor1946theory} is defined as the collection of $L^2$-functions
\begin{equation}\label{eq: Gabor system}
    \G(g, \alpha, \beta) := \left\{ e^{2\pi i \beta m (\cdot)} g(\cdot - \alpha n) = M_{\beta m} T_{\alpha n} g \right\}_{n,m \in \mathbb{Z}},    
\end{equation}
where $T_a$ denotes the translation operator by $a \in \mathbb{R}$, i.e., $T_a f(x) := f(x - a)$, and $M_b$ denotes the modulation operator by $b \in \mathbb{R}$, i.e., $M_b f(x) := e^{2\pi i b x} f(x)$. 
This structure has particular importance in many applications such as wireless communication, analysis and description of speech and music signals, among others (see, for example, \cite{grochenig2013foundations,strohmer2003grassmannian}). Complementary work on the mathematical structure of Gabor frames includes the study of frame sets for specific compactly supported windows \cite{lemvig2016counterexamples}, the stability of Gabor frame bounds under changes in the underlying time-frequency lattice \cite{grochenig2023lipschitz,grochenig2024holder}, and the optimization of frame bounds and conditioning over classes of lattices \cite{faulhuber2023gabor}.

For $\G(g, \alpha, \beta)$ to be a frame for $L^2(\mathbb{R})$, it is necessary (but not sufficient) that $\alpha \beta \leq 1$, as well as the condition
\begin{equation}\label{eq: implies bounded window}
    A \leq \frac{1}{\beta} \sum_{k\in \mathbb{Z}} \left| g(x +\alpha k) \right|^2 \leq B \qquad \text{for a.e. } x\in \R,    
\end{equation}
where $A$ and $B$ are the frame bounds (see, for example, \cite{grochenig2013foundations}).

In this work, given a Gabor frame $\G(g, \alpha, \beta)$ for $L^2(\mathbb{R})$, we examine its dual frames that are also Gabor systems of the form $\G(h, \alpha, \beta)$ for some $h\in L^2(\R)$. To that end, let us consider the operator
\begin{equation}
    S_{g,h}(f) := \sum_{m,n \in \mathbb{Z}} \langle f, M_{m\beta} T_{\alpha n} g \rangle M_{m\beta} T_{\alpha n} h = \sum_{m,n \in \mathbb{Z}} \langle f, T_{\alpha n} M_{m\beta} g \rangle T_{\alpha n} M_{m\beta} h,
\end{equation}
where the second equality holds because $T_a M_b f = e^{-2\pi i b a} M_b T_a f$.
In particular, $S = S_{g,g}$ is the \textit{frame operator} of $\G(g, \alpha, \beta)$. It is well-known that if $\G(g, \alpha, \beta)$ is a frame for $L^2(\mathbb{R})$, then its canonical dual frame has also a Gabor structure $\G(\widetilde{g}, \alpha, \beta)$ \cite[Thm 12.3.2]{christensen2003introduction}, where $\widetilde{g} := S_{g,g}^{-1} g$ is called the \textit{canonical dual window}.
In general, $\G(g, \alpha, \beta)$ and $\G(h, \alpha, \beta)$ are dual frames if and only if $S_{g,h} = I = S_{h,g}$ and, in this case, we say that $g$ and $h$ are \textit{dual Gabor windows}, or equivalently, that $h$ is a \textit{dual Gabor window} for $g$.

Given $g,h\in L^2(\R)$ and $\alpha,\beta>0$, for $n\in\Z$ define, whenever the series converges, the correlation function (cross-Gramian matrix)
$$G_n^{g,h}(x):=\sum_{k\in\Z} \overline{g(x+\tfrac{n}{\beta}+\alpha k)}h(x+\alpha k) ,\qquad n\in\Z.$$
For example, this series is well defined under suitable assumptions such as $g,h\in W(\R)$, where $W(\R)$ denotes the \textit{Wiener amalgam space}  \cite[Sec. 6.1]{grochenig2013foundations}.
Then, the operator $S_{g,h}$ can be written through the Walnut representation formula
\begin{equation}\label{eq: Walnut}
    S_{g,h}f(x)=\frac{1}{\beta}\sum_{n\in \Z}G_n^{g,h}(x)f(x+\tfrac{n}{\beta})
\end{equation}
See, for example \cite{walnut1992continuity} and \cite[Sec. 11.5 $\&$ Thm 12.2.1]{christensen2003introduction}. 
Independently, Janssen in \cite[Eq. (1.3.20)]{janssen1998duality} and Ron and Shen in \cite[Corollary 3.8]{ron1997weyl}, \cite[Sec. 4]{ron1995frames} provided the following general characterization of Gabor frames: 
Two Bessel sequences $\mathcal G(g,\alpha,\beta)$, $\mathcal G(h,\alpha,\beta)$ are dual Gabor frames for $L^2(\R)$ if and only if
\begin{equation}\label{eq: characterization dual}
  \sum_{k\in\Z} \overline{g(x+\tfrac{n}{\beta}+ k\alpha)}h(x+ k\alpha)=\beta\delta_{n,0} \qquad \text{ for a.e. } x\in[0,\alpha),
\end{equation}
that is, $G_0^{g,h}(x)=\beta \text{ and } \, G_n^{g,h}(x)=0 \, (n\neq 0)$ for a.e. $x\in [0,\alpha)$.

It is well known that the canonical dual window minimizes the $L^2$-norm among all dual Gabor windows. Specifically, Daubechies, Landau, and Landau showed in \cite{daubechies1994gabor} that if $\G(g,\alpha,\beta)$ is a Gabor frame for $L^2(\R)$, its canonical dual window $\widetilde{g}$ satisfies
$$\|\widetilde{g}\|_2\leq \|h\|_2 \qquad \text{ for every dual Gabor window } h.$$

The problem of selecting dual Gabor windows with additional optimality or localization properties has also been studied from sparse and computational perspectives. In the finite-dimensional setting, Li, Liu, and Mi investigate sparse dual frames and dual Gabor functions having minimal time and/or frequency support and propose $\ell^1$-minimization procedures for their construction \cite{li2013sparse}. Perraudin, Holighaus, S{\o}ndergaard, and Balazs develop a convex-optimization framework for designing dual Gabor windows subject to constraints promoting desirable properties such as smoothness, time-frequency localization, sparsity, and compact support \cite{perraudin2013gabor}. In contrast, the present work considers the problem of minimizing the $L^1$-norm over all exact dual Gabor windows and provides explicit analytic characterizations in the regimes considered below.

\subsection{Contributions} In this work, given a Gabor frame $\G(g,\alpha,\beta)$, we study the problem of characterizing dual Gabor windows with minimal $L^1$-norm among all dual Gabor windows for $g$. \\

In Sections \ref{sec: painless} and \ref{sec: painful}, we assume that $g$ is compactly supported, with $$\supp(g)\subseteq [-L/2,L/2]$$ for some fixed length $L>0$, and  separate our analysis into two regimes: the painless non-orthogonal case, in which the frame operator $ S_{g,g} $ is 
a multiplication operator and therefore explicitly invertible under the frame condition,
which occurs when $ 0 < \alpha \leq L $ and $ 0 < \beta \leq 1/L $ (Section \ref{sec: painless}); and the painful regime with $ \alpha = L $ and $ \beta > 1/L $, with particular emphasis on the normalized choice $\alpha=1$ and $\beta=2/3$ (Section \ref{sec: painful}).
\\

For readability, we make the following
{assumptions in Sections \ref{sec: painless} and \ref{sec: painful}}:
\begin{enumerate}
    \item     \textit{Support of the window:} 
    First, throughout the paper, functions in $L^p(\mathbb R)$ are identified up to equality almost everywhere. We say that $f$ is supported on a measurable set $E$ if $f=0$ almost everywhere on $\mathbb R\setminus E$. Accordingly, all statements concerning supports are understood in the almost-everywhere sense. 

    Second, without loss of generality, it suffices to consider windows $g$ satisfying $\supp(g)\subset[-1,1]$. The general case $\supp(g)\subset[-L/2,L/2]$ can be reduced to this by defining a rescaled window $x\mapsto \sqrt{L/2}\, g(Lx/2)$, which induces rescaled Gabor parameters $\alpha' = \frac{2\alpha}{L}$ and $\beta' = \frac{\beta L}{2}$. 
   
    \item \textit{Bounded window:} Let $g\in L^2(\R)$ compactly supported, and consider the Gabor system  $\G(g,\alpha,\beta)$ as in \eqref{eq: Gabor system}. On the one hand, as a consequence of \eqref{eq: implies bounded window}, if $ \G(g, \alpha, \beta) $ forms a Gabor frame, then  
    $$
    |g(x)|^2 \leq \sum_{k\in \mathbb{Z}} |g(x + \alpha k)|^2 \leq \beta B \qquad \text{for a.e. } x\in \R.
    $$  
    implying that $ g \in L^\infty(\mathbb{R}) $ (see also \cite[Theorem 3.12]{benedetto1993walnut}). On the other hand, it is known that if $ g $ is bounded and compactly supported, then $ \G(g, \alpha, \beta) $ is automatically a Bessel sequence for any $ \alpha, \beta > 0 $ (see, for example,  \cite[Corollary 11.4.3]{christensen2008frames}). Due to these observations, we will assume throughout this work that $ g $ is bounded, i.e., $g\in L^\infty(\R)$. 
\end{enumerate}

Finally, in Section \ref{sec: diag dominant}  we weaken the assumption of compact support on the window $g$ by considering general Gabor frames with diagonally dominant frame operator and study \textit{approximate dual windows} with small $L^1$-norm.
\\

While in general one may allow $g:\R\to \C$, the complex-valued setting plays no special role in our arguments. For this reason, throughout the paper we adopt the convention that $g$ is real-valued.\\

Our main results can be summarized as follows:
\begin{itemize}[leftmargin=*] 
    \item In the painless regime, we prove that any dual Gabor window minimizing the $L^1$-norm must have support of measure at least $\alpha$, and that one can always choose such a minimizer supported inside the support of the original window. 
    We then give an explicit construction of an $L^1$-minimizing dual window by selecting, for almost every $x$ in a fundamental domain, the translate on which $|g|$ is maximal and defining the dual by $h=\beta/g$ on the corresponding set (see Theorem \ref{thm: main thm painless}). Furthermore, we characterize when this minimizer is unique: uniqueness holds whenever the values $|g(x+\alpha k)|$ are pairwise distinct for almost every $x$.
    In the non-unique case, we completely describe the full family of $L^1$-minimizing dual windows (see Theorem \ref{thm: allL1minimizers}) and  identify exactly those minimizers that also have minimal possible support measure, commonly referred to, by a slight abuse of terminology, as the $L^0$-norm (see Corollary \ref{coro: main corollary painless support}).

    \item In the painful regime, we focus on the case $\alpha=1$ and $\beta=2/3$. Using a compact-support reduction for dual windows together with a pointwise finite-dimensional $\ell^1$ minimization problem, we derive an explicit description of the dual window with minimal $L^1$-norm in this case (see Theorem \ref{thm: main painful}).
    
    \item Finally, for general Gabor frames with diagonally dominant frame operator, we consider approximate duality. There, we construct compactly supported approximate dual windows with nearly minimal $L^1$-norm and show in particular that for windows in a suitable decay class, one obtains approximate duals whose $L^1$-norm approaches the natural lower bound when the lattice parameters are sufficiently small (see Theorem \ref{thm: diag dom}).
\end{itemize}

In order to make the exposition as clear and accessible as possible, each section is accompanied by examples.
Moreover, we include a discussion of an apparent contradiction in the literature concerning the existence of compactly supported canonical dual windows beyond the painless regime (see Section \ref{sec: discussion}).

\section{Painless Non-orthogonal Expansions}\label{sec: painless}

For the entirety of this section, let $g \in L^\infty(\R)$ be a function of compact support, say on the interval $[-1,1]$. Then, for $0 < \alpha \leq 2$ and $0 < \beta \leq 1/2$, the frame operator associated to the Gabor system $\G(g,\alpha,\beta)$ is given by multiplication by 
\begin{equation*}
\tfrac{1}{\beta}G_0^{g,g}(x) = \tfrac{1}{\beta}\sum_{k \in \Z} |g(x+\alpha k)|^2. 
\end{equation*}
This was originally shown in \cite{daubechies1986painless}, but it can also be deduced from Walnut's representation of the frame operator \cite{walnut1992continuity}. It follows immediately that $\G(g,\alpha,\beta)$ is a frame for $L^2(\R)$ if and only if $G_0^{g,g}(x)$ is bounded and bounded away from zero. That is, there exist $B\geq A > 0$ such that 
\begin{equation}\label{eq: bounds of G0gg}
    A \leq \tfrac{1}{\beta} G_0^{g,g}(x) \leq B
\end{equation}
holds for almost all $x \in \R$. See also \cite[Theorem 6.4.1]{grochenig2013foundations}.

We will prove the following theorem. 
\begin{theorem}\label{thm: main thm painless}
    Let $\mathcal{G}(g,\alpha,\beta)$ be a Gabor frame with $g \in L^\infty(\R)$, $\supp{(g)} = [-1,1]$, and $\alpha \leq 2$, $\beta \leq 1/2$. Then there is an explicit set $I \subset [-1,1]$ (see Construction \ref{cons: set_I}) such that
    \begin{equation*}
   h(x) =  \begin{cases}
       \frac{\beta}{ {g(x)}}
 & x \in I \\
 0 & x \notin I
 \end{cases}
    \end{equation*}
    is a dual window to $\mathcal{G}(g,\alpha,\beta)$ that has minimal $L^1$-norm. Furthermore, if for almost every $x \in [0,\alpha)$, the values $|g(x+\alpha k)|$ are all distinct for 
    \begin{equation*}
    k \in \Z \cap \left[\tfrac{x-1}{\alpha}, \tfrac{x+1}{\alpha} \right],
    \end{equation*}
 then $h(x)$ is the unique dual window of $\mathcal{G}(g,\alpha,\beta)$ with minimal $L^1$-norm. 
\end{theorem}

\begin{remark}\label{rmk: canonical dual not L1 min}
Although the canonical dual window minimizes the $L^2$-norm among all dual Gabor windows, it does not, in general, minimize the $L^1$-norm. Example \ref{example: cosine} below gives a case in which the $L^1$-norm of the canonical dual is strictly larger than the minimum. By contrast, Example \ref{example: non !} shows that the canonical dual may be $L^1$-minimizing without being the unique $L^1$-minimizer or having minimal possible support ($L^0$-norm).
\end{remark}

We begin by establishing some preliminary results.  We will use the notation 
$\mathbf{1}_A$ to denote the indicator function of the measurable set $A$.

\begin{proposition}\label{P:support}
    Let $g\in L^\infty(\R)$ be supported on $[-1,1]$. Assume $\mathcal{G}(g,\alpha,\beta)$ is a Gabor frame and $\mathcal{G}(h,\alpha,\beta)$ is an arbitrary dual Gabor frame. Then the support of $h$ has Lebesgue measure at least $\alpha$. If, in addition, $\alpha \leq 2$, $\beta \leq 1/2$ and $h$ is such that $\|h\|_1$ is minimized among all dual Gabor windows, then $\supp(h) \subset [-1,1]$. 
\end{proposition}

\begin{proof}[Proof of Proposition \ref{P:support}] 
Let $E:=\supp(h)$ and assume, toward a contradiction, that
$
|E|=\alpha-\varepsilon<\alpha
$,
for some $\varepsilon>0$.
Consider the $\alpha$-periodization of $E$ on a fundamental domain,
\begin{equation*}
F:=\bigcup_{k\in\mathbb Z}(E-\alpha k)\cap [0,\alpha).
\end{equation*}
Thus $F$ consists of those points $x\in[0,\alpha)$ for which there exists
$k\in\mathbb Z$ such that $x+\alpha k\in E$. 
To show that $|F|\le |E|$, define
\begin{equation*}
E_k:=E\cap [k\alpha,(k+1)\alpha), \qquad k\in\mathbb Z.
\end{equation*}
Then the sets $E_k$ are pairwise disjoint and
$
E=\bigcup_{k\in\mathbb Z} E_k$.
Moreover, $
F=\bigcup_{k\in\mathbb Z}(E_k-\alpha k)$.
Since each set $E_k-\alpha k\subset[0,\alpha)$ and translation preserves Lebesgue measure, i.e., 
$|E_k-\alpha k|=|E_k|$ for all $k\in\mathbb Z$, we obtain, 
by subadditivity of Lebesgue measure, that $|F|\leq |E|<\alpha$. 
Therefore, the set
\begin{equation*}
U:=[0,\alpha)\setminus F
\end{equation*}
has positive measure and satisfies that for every $x\in U$ and every $k\in\mathbb Z$ we have
$x+\alpha k\notin E=\supp(h)$, hence $h(x+\alpha k)=0$ $\forall k\in\mathbb Z$.
Consequently,
\begin{equation*}
\sum_{k\in\mathbb Z} {g(x+\alpha k)}\,h(x+\alpha k)=0\neq \beta
\qquad \forall x\in U,
\end{equation*}
which contradicts \eqref{eq: characterization dual} for $n=0$.
Hence $|\supp(h)|=|\{x\in\mathbb R: \,  h(x)\not=0\}|\ge \alpha$.

For the second claim, when $n \neq 0$, equation \eqref{eq: characterization dual} is trivially satisfied by all functions that have support in $[-1,1]$. Furthermore, when $n = 0$, equation \eqref{eq: characterization dual} depends only on the values of $h$ in the interval $[-1,1]$.
Hence, if $h$ is a dual window to $\mathcal{G}(g,\alpha,\beta)$, then  $h\mathbf{1}_{[-1,1]}$ satisfies the same duality equations \eqref{eq: characterization dual} as $h$. Also, since $\mathcal{G}(h, \alpha, \beta)$ is a Bessel sequence, $h \in L^{\infty}(\mathbb{R})$. So, $h \mathbf{1}_{[-1,1]}$ is bounded and compactly supported, and therefore $\mathcal{G}\left(h \mathbf{1}_{[-1,1]}, \alpha, \beta\right)$ is a Bessel sequence.  Consequently, $\mathcal{G}(g, \alpha, \beta)$ and $\mathcal{G}\left(h \mathbf{1}_{[-1,1]}, \alpha, \beta\right)$ are dual Gabor frames. (We will use this same argument later in Section \ref{sec: painful}.) Finally, since $\|h\mathbf{1}_{[-1,1]}\|_1 \leq \|h\|_1$ with equality if and only if $h=0$ almost everywhere outside $[-1, 1]$, every $L^1$-minimizing dual window is supported on $[-1,1]$. This completes the proof. (We notice that the same argument holds for $1\leq p<\infty$, that is, it is not particular to $L^1$-minimizer duals.)
\end{proof}

Proposition \ref{P:support} demonstrates that we can narrow the search for dual windows with minimal $L^1$-norm to only those dual windows which are supported on $[-1,1]$. As it turns out, this restriction makes the search quite simple. However, we will need some preparatory results. For the remainder of this section, let 
\begin{equation}\label{eq: N alpha}
    N_\alpha :=\left\lceil{\tfrac{{2}}{\alpha}}\right\rceil
\end{equation}
so that for almost all $x\in[0,\alpha]$, only finitely many integers $k$
satisfy $x+\alpha k\in[-1,1]$, and all such $k$ lie in $\{-N_\alpha,\dots,{N_\alpha}\}$.

\begin{lemma}\label{L:well-defined}
    If $\mathcal{G}(g,\alpha,\beta)$ is as in Theorem \ref{thm: main thm painless}, then for almost all $x \in [0,\alpha]$ 
    \begin{equation*}
    \max_{-N_\alpha \leq k \leq {N_\alpha}}{|g(x+\alpha k)|} > 0,
    \end{equation*}
    where $N_\alpha$ is defined by \eqref{eq: N alpha}.
\end{lemma}

\begin{proof}[Proof of Lemma \ref{L:well-defined}]
If $0 = g(x+\alpha k)$ for each $k \in \{-N_\alpha,...,{N_\alpha}\}$ holds in some set $U$ of positive measure, then 
\begin{equation*}
G_0^{g,g}(x) = \sum_{k = -N_\alpha}^{{N_\alpha}} |g(x+\alpha k)|^2 = 0
\end{equation*}
for all $x \in U$. This contradicts the assumption that $\mathcal{G}(g,\alpha,\beta)$ is a frame. 
\end{proof}

\begin{lemma}\label{L:bound}
    Let $\G(g,\alpha,\beta)$ be as in Theorem \ref{thm: main thm painless}. If $h$ is a dual window supported on $[-1,1]$, then 
    \begin{equation}\label{E:lower_bound}
    \sum_{k = -N_\alpha}^{{N_\alpha}} |h(x+\alpha k)| \geq \frac{\beta}{\max\limits_{-N_\alpha\leq k\leq {N_\alpha}}|g(x+\alpha k)|},
    \end{equation}
    where $N_\alpha$ is defined by \eqref{eq: N alpha}.
    If equality in \eqref{E:lower_bound} is achieved for almost all $x \in [0,\alpha]$, then $h$ has minimal $L^1$ norm among all dual windows for $\mathcal{G}(g,\alpha,\beta)$.
\end{lemma}

\begin{proof}[Proof of Lemma \ref{L:bound}]
    Since $h$ is assumed to be a dual window, by \eqref{eq: characterization dual} we have 
    \begin{equation}\label{eq: > beta}
        \beta = \bigg|\sum_{k=-N_\alpha}^{{N_\alpha}}  {g(x+\alpha k)}h(x+\alpha k)\bigg| \leq \max\limits_{-N_\alpha\leq k\leq {N_\alpha}}|g(x+\alpha k)|\cdot\sum_{k=-N_\alpha}^{{N_\alpha}} |h(x+\alpha k)|    
    \end{equation}
    for almost all $x \in [0,\alpha]$. This gives the lower bound 
    \begin{equation*}
    \sum_{k = -N_\alpha}^{{N_\alpha}} |h(x+\alpha k)| \geq \frac{\beta}{\max\limits_{-N_\alpha\leq k\leq {N_\alpha}}|g(x+\alpha k)|}
    \end{equation*}
    which is well-defined for almost all $x$ by Lemma \ref{L:well-defined}. To show the second claim, we only have to note that 
     \begin{align*}
        \|h\|_1 = \int_{{-1}}^1 |h(x)|\;dx = \int_0^\alpha \sum_{k = -N_\alpha}^{{N_\alpha}} |h(x+\alpha k)|\;dx.
    \end{align*}
    Hence if the $\alpha$-periodization of $h$ is minimized pointwise, the $L^1$-norm of $h$ is minimized. 
\end{proof}

\begin{remark}
    We highlight an important takeaway from the proof of Lemma \ref{L:bound}: Given a dual window $h$ that minimizes the $L^p$-norm over all possible dual windows of $g$, we have 
    $$\|h\|_p^p=\int_0^\alpha\sum_{k=-N_\alpha}^{N_\alpha}|h(x+\alpha k)|^p \, dx,$$
    and the problem of finding an $L^p$-minimizing dual, for any $1\leq p<\infty$, reduces to minimizing an $\ell^p$ objective over the vector of lattice translates  $(h(x-\alpha N_\alpha),\dots, h(x+\alpha N_\alpha))$ for a.e. $x\in[0,\alpha]$. 
    The constraints to be satisfied are given by \eqref{eq: characterization dual}. Since $\supp{(h)} \subseteq [-1,1]$ (Proposition \ref{P:support}), the only nontrivial equation in \eqref{eq: characterization dual} is the $n=0$ condition, which reads as follows for a.e. $x\in[0,\alpha]$: 
    $$
    \beta=\sum_{k\in\Z}  {g(x+\alpha k)}\,h(x+\alpha k)=\sum_{k=-N_\alpha}^{N_\alpha}  {g(x+\alpha k)}\,h(x+\alpha k).
    $$
    Thus, for each fixed $x\in[0,\alpha]$ we are solving a finite-dimensional  $\ell^p$ minimization problem with one linear constraint.
\end{remark}

We now prepare to prove Theorem \ref{thm: main thm painless}.  

\begin{construction}\label{cons: set_I}
We construct the set $I \subset [-1,1]$ in the statement of Theorem \ref{thm: main thm painless} (well-defined up to a set of measure 0) that will serve as the support of our dual window. By Lemma \ref{L:well-defined}, there is a measurable set $U \subseteq [0,\alpha)$ of Lebesgue measure $\alpha$ such that 
\begin{equation}\label{eq: max}
    M_\alpha(x):=\max_{-N_\alpha \leq k \leq {N_\alpha}}{|g(x+\alpha k)|} > 0
\end{equation}
for all $x \in U$. Given $x \in U$, let $J_x \subset \{-N_\alpha,...,{N_\alpha}\}$ be the set of indices such that the above maximum \eqref{eq: max} is achieved, i.e.
\begin{equation}\label{def: J_x}
    J_x = \argmax_{-N_\alpha\leq k \leq {N_\alpha}} |g(x+\alpha k)|.    
\end{equation}
Then, we define the set $I$ as follows:
\begin{center}
\textit{If $k_x$ is the smallest index in $J_x$, we say $x+\alpha k_x\in I$ and $x+\alpha\ell \notin I$ for all integers $\ell \neq k_x$. }
\end{center}
Thus, $I \subset [-1,1]$ is well-defined up to a subset of measure 0, and has the property that for almost all $x \in \R$, there is exactly one integer $k \in \Z$ such that $x+\alpha k \in I$. 

Moreover, it follows that $I$ has Lebesgue measure $\alpha$:
the assignment $\phi:U\to I$, $\phi(x):=x+\alpha k_x$ is well-defined, and one can partition $$U=\bigcup_{k=-N_\alpha}^{{N_\alpha}}\underbrace{\{x\in U\mid \, k_x=k\}}_{U_k}.$$
Since $g$ is measurable, the finite maximum $M_\alpha(x)$, defined by \eqref{eq: max},
is measurable. Moreover, the tie-breaking rule defining $k_x$ as the
smallest index at which this maximum is attained gives $U_k
=
U\cap \left\{x:\ |g(x+\alpha k)|=M_\alpha(x)\right\}$.
Hence, each $U_k$ is measurable. 
On each $U_k$ the map $\phi$ acts by translation by $\alpha k$, i.e.,  $\phi(U_k)=U_k+\alpha k\subset [\alpha k, \alpha(k+1))$. So, each $\phi(U_k)$ is measurable and we have a new disjoint union $I=\phi(U)=\bigcup_{k=-N_\alpha}^{{N_\alpha}}\phi(U_k)$, with $|\phi(U_k)|=|U_k|$. The  result is that
$$|I|=\sum_{k=-N_\alpha}^{{N_\alpha}}|\phi(U_k)|=\sum_{k=-N_\alpha}^{{N_\alpha}}|U_k|=|U|=\alpha.$$
\end{construction}

\begin{proof}[Proof of Theorem \ref{thm: main thm painless}]
    First, we note that $h(x)$ is well-defined up to a set of measure 0. Indeed, the set $I$ is defined up to a set of measure 0, and $|g(x)| > 0$ for all $x \in I$ by construction, so $h$ is well-defined. 
    
    Now we show that $h$ is a dual window. For almost all $x \in [-1,1]$, there exists some $\ell \in \Z$ such that $x + \alpha \ell \in I$ and $x + \alpha k \notin I$ for $k \neq \ell$. It follows that 
    \begin{equation*}
    \sum_{k\in \Z}  {g(x+\alpha k)}h(x+\alpha k) =  {g(x+\alpha \ell)}\frac{\beta}{ {g(x+\alpha \ell)}} = \beta.
    \end{equation*}
    Furthermore, since $\supp{(h)} = I \subset [-1,1]$, we have 
    \begin{equation*}
    \sum_{k\in\Z} {g(x+\tfrac{n}{\beta}+ k\alpha)}h(x+ k\alpha)=0
    \end{equation*}
    for all $n \neq 0$ and almost all $x \in [0,\alpha)$. We have thus shown that $h$ satisfies $\eqref{eq: characterization dual}$.

 To conclude that $h(x)$ is a dual window, we need that $\mathcal{G}(h,\alpha,\beta)$ is a Bessel sequence. Observe that for all $x \in I$, 
    \begin{equation*}
    |h(x)|^2 = \frac{\beta^2}{|g(x)|^2}
    \end{equation*}
    and $|g(x)|^2 \geq |g(x+\alpha k)|^2$ for all $0\neq k\in \Z$ by construction. Indeed, if $x\in I$, then $x=u+\alpha k_u$ for some
$u\in U$, where $k_u$ is chosen so that
$
|g(x)|=|g(u+\alpha k_u)|=\max_j |g(u+\alpha j)|=\max_k |g(x+\alpha k)|$.  
    Hence, using that the number of nonzero terms in $\sum_{k\in\Z} |g(x+\alpha k)|^2$ is at most $2N_\alpha+1$, and using \eqref{eq: bounds of G0gg}, we have
    \begin{equation}\label{eq: for the bessel bound}
            |h(x)|^2 = \frac{\beta^2}{|g(x)|^2} \leq \frac{\beta^2}{\frac{1}{2N_\alpha+1}\sum_{k\in\Z} |g(x+\alpha k)|^2} = \frac{(2N_\alpha+1)\beta^2}{G_0^{g,g}(x)} \leq \frac{(2N_\alpha+1)\beta^2}{\beta A}.
    \end{equation}
    So, $h \in L^\infty(\R)$ and is compactly supported. As mentioned in Section \ref{sec: Intro}, these properties on $h$ guarantee that $\mathcal{G}(h,\alpha,\beta)$ is a Bessel sequence (see, for example, \cite[Corollary 11.4.3]{christensen2008frames}). 

    Finally, we argue that $h(x)$ has minimal $L^1$ norm among all dual windows of $\mathcal{G}(g,\alpha,\beta)$. By Lemma \ref{L:bound}, we only need to argue that the $\alpha$-periodization of $h$ achieves equality in \eqref{E:lower_bound} for almost all $x$. 
    Indeed, by the defining property of $I$ in Construction \ref{cons: set_I},
for a.e. $x\in[0,\alpha)$, there exists exactly one $\ell\in\mathbb Z$
such that $x+\alpha\ell\in I$. Since $h=\beta/g$ on $I$ and $h=0$
on $\mathbb R\setminus I$, this is equivalent to saying that there exists
exactly one $\ell\in\mathbb Z$ such that $h(x+\alpha\ell)\neq 0$.
 Therefore, 
    \begin{equation*}
    \sum_{k  = -N_\alpha}^{{N_\alpha}} |h(x+\alpha k)| = |h(x+\alpha \ell)| = \frac{\beta}{|g(x+\alpha \ell)|} = \frac{\beta}{\max\limits_{-N_\alpha\leq k\leq {N_\alpha}}|g(x+\alpha k)|}.
    \end{equation*}

Finally, the statement on uniqueness will follow from the more general Theorem \ref{thm: allL1minimizers}.
\end{proof}

In our current construction,  we defined $k_x$ as the smallest index in $J_x$ (see \eqref{def: J_x}). This choice is arbitrary whenever $J_x$ contains more than one index, and only serves to produce a particular minimizer. Varying such a choice gives rise to all possible $L^1$ minimizers in the non-unique case. The following theorem  characterizes all possible $L^1$-minimizing duals.

\begin{theorem}\label{thm: allL1minimizers}
Let $\mathcal{G}(g,\alpha,\beta)$ be a Gabor frame with $\supp(g)=[-1,1]$ and $\alpha\leq 2$, $\beta\leq 1/2$.
Let $J_x$ be as in \eqref{def: J_x}. Then, up to sets of measure zero, 
a dual window $h$ is $L^1$-minimizing if and only if there exists a
measurable family of weights $w(x)=(w_k(x))_{k\in J_x}$ with
\begin{equation*}\label{eq: weight}
w_k(x)\geq 0 \quad \text{and} \quad
\qquad
\sum_{k\in J_x} w_k(x)=1,
\end{equation*}
such that 
\begin{equation}\label{def: general h}
h(x+\alpha k)
=
\begin{cases}
\frac{\beta\, w_k(x)}{ {g(x+\alpha k)}}, & k\in J_x,\\ 
0, & k\notin J_x.
\end{cases}
\end{equation}
\end{theorem}

Note that, in particular, Theorem \ref{thm: main thm painless} corresponds to the choice
\begin{equation*}
w_k(x)=
\begin{cases}
1, & k=k_x\in J_x,\\
0, & k\neq k_x.
\end{cases}
\end{equation*}
Furthermore, if $J_x$ is a singleton set for almost all $x$, then this is the only admissible family of weight functions $w_k$, from which the uniqueness statement of Theorem \ref{thm: main thm painless} follows.

\begin{proof}[Proof of Theorem \ref{thm: allL1minimizers}]
Let $h$ be defined by \eqref{def: general h}. Since $h(x+\alpha k)\not
=0$ only when $k\in J_x$,  and recall that $\supp(h)\subset [-1,1]$. 
    Fix $x\in[0,\alpha)$, then $k\in J_x$, implies
$x+\alpha k \in \supp(g)\subseteq [-1,1]$. Therefore, as in the proof of Theorem \ref{thm: main thm painless}, all duality conditions in \eqref{eq: characterization dual} for $n\neq 0$
are automatically satisfied, and it only remains to verify the case
$n=0$.
    By defining $h$ as in \eqref{def: general h}, we automatically obtain
\begin{equation}
    \sum_{k}  {g(x+\alpha k)}\, h(x+\alpha k)
    =
    \sum_{k\in J_x}  {g(x+\alpha k)}
    \frac{\beta}{ {g(x+\alpha k)}} w_k(x)
    =
    \beta \sum_{k\in J_x} w_k(x)
    =
    \beta.    
\end{equation}
The Bessel property follows as an adaptation of the argument \eqref{eq: for the bessel bound} in Theorem \ref{thm: main thm painless}. Indeed, 
$$\frac{\beta^2w_k(x)^2}{|g(x+\alpha k)|^2}
\leq
\frac{\beta^2}
{\max\limits_{-N_\alpha\leq j\leq N_\alpha}|g(x+\alpha j)|^2}
\leq
\frac{(2N_\alpha+1)\beta}{A},
$$
where $A$ is a lower frame bound for $\mathcal{G}(g,\alpha,\beta)$. Thus $h\in L^\infty(\mathbb{R})$, and since $\supp(h)\subseteq[-1,1]$, the Gabor system $\mathcal{G}(h,\alpha,\beta)$ is Bessel. Therefore, the identities above imply that $h$ is a dual Gabor window.
Moreover,
\begin{align*}
\sum_{k} |h(x+\alpha k)|
=
\beta
\sum_{k\in J_x} 
\frac{w_k(x)}{|g(x+\alpha k)|}=\frac{\beta}{\max\limits_{-N_\alpha\leq k\leq N_\alpha}|g(x+\alpha k)|}.    
\end{align*}
Thus, every such choice of weights $w$ produces an $L^1$-optimal solution, as the corresponding $h$ given by \eqref{def: general h} achieves equality in \eqref{E:lower_bound} for almost all $x$ (see Lemma \ref{L:bound}).

To prove the converse, let $h$ be an $L^1$-minimizing dual window. By Proposition \ref{P:support}, we  assume
$\supp(h)\subset[-1,1]$, and so the duality condition \eqref{eq: characterization dual} needs to be checked only for $n=0$, that is, gives
\begin{equation*}
    \sum_{k=-N_\alpha}^{N_\alpha}  {g(x+\alpha k)} h(x+\alpha k) = \beta  \qquad a.e. \, x\in [0,\alpha].
\end{equation*}
Lemma \ref{L:bound} implies the pointwise estimate
\begin{equation}\label{eq: geq}
    \sum_{k=-N_\alpha}^{N_\alpha} |h(x+\alpha k)|
    \geq
    \frac{\beta}{\max\limits_{-N_\alpha\leq k\leq N_\alpha} |g(x+\alpha k)|}.    
\end{equation}
Since $h$ minimizes $\|h\|_1$ (and we already proved in Theorem \ref{thm: main thm painless} the existence of a dual window reaching the equality in \eqref{eq: geq}), then equality in \eqref{eq: geq} must hold for almost every
$x\in[0,\alpha]$; otherwise the integral defining $\|h\|_1$ could be
strictly decreased. At this point, after having identified sufficient conditions that should be satisfied for a.e. $x\in [0,\alpha]$, let us simplify our notation:
Fix $x\in[0,\alpha]$, for $-N_\alpha\leq k\leq N_\alpha$ let 
\begin{equation*}
    g_k := g(x+\alpha k), \qquad h_k := h(x+\alpha k),
\qquad
m := \max_{-N_\alpha\leq k\leq N_\alpha} |g_k|.    
\end{equation*}
 Hence $h$ is characterized by 
\begin{equation}\label{eq: h_k}
    \sum_{k=-N_\alpha}^{N_\alpha}  {g_k}h_k=\beta, \qquad \text{and} \qquad \sum_{k=-N_\alpha}^{N_\alpha} |h_k|
=
\frac{\beta}{m}  .
\end{equation}
Observe that
\begin{equation*}
\beta
=
\left|\sum_{k=-N_\alpha}^{N_\alpha}  {g_k} h_k\right|
\leq
\sum_{k=-N_\alpha}^{N_\alpha} |g_k||h_k|
\leq
m\sum_{k=-N_\alpha}^{N_\alpha} |h_k|.    
\end{equation*}
Equality must hold in the
inequalities above due to the second condition in \eqref{eq: h_k}. Consequently, since $|g_k|\leq m$, we have that
$h_k=0$ whenever $|g_k|<m$, hence $h_k=0$ for $k\notin J_x$. Additionally, for $k\in J_x$,
the equality in the triangle inequality (i.e., $\left|\sum_{k=-N_\alpha}^{N_\alpha}  {g_k} h_k\right|
=
\sum_{k=-N_\alpha}^{N_\alpha} | {g_k}h_k|$) implies that the values $ {g_k} h_k$ have the same sign. Since their sum equals $\beta$ (due to the first condition in \eqref{eq: h_k}), we have  $ {g_k} h_k\geq 0$.
Defining
\begin{equation*}
w_k(x) := \frac{ {g_k} h_k}{\beta}, \qquad k\in J_x,
\end{equation*}
we have $w_k(x)\ge0$, $
\sum_{k\in J_x} w_k(x) = 1$, and 
$h_k = \frac{\beta w_k(x)}{ {g_k}}$ (for  $k\in J_x$), 
which is precisely the form \eqref{def: general h}.
\end{proof}

\begin{corollary}\label{coro: main corollary painless support}
    Let $\mathcal{G}(g,\alpha,\beta)$ be a Gabor frame with $\supp{(g)} = [-1,1]$ and $\alpha \leq 2$, $\beta \leq 1/2$.
    All $L^1$-minimizing duals $h$ that also have support of minimal possible measure, that is, $|\supp(h)|=\alpha$ are of the form \eqref{def: general h} with weight vector $w(x)$ satisfying
$w_k(x)\in\{0,1\}$
(i.e., with exactly one index $k\in J_x$ such that $w_k(x)=1$). 
\end{corollary}

In essence, from all $L^1$-minimizing dual windows characterized by Theorem \ref{thm: allL1minimizers}, only those corresponding to choices 
of weights $w$ that
are vertices of the simplex
$\Delta_x:=\{ w(x)=(w_k(x))_{k\in J_x}: \,
w_k(x)\geq 0,
\,
\sum_{k\in J_x} w_k(x)=1\}$
for almost every $x$,
produce a dual window with minimal possible support.

\begin{proof}[Proof of Corollary \ref{coro: main corollary painless support}]
Theorem \ref{thm: main thm painless} constructs an $L^1$-minimizing dual window supported on  $I\subset[-1,1]$ with $|I|=\alpha$, proving the
existence of a minimizer with support measure $\alpha$.
By Theorem \ref{thm: allL1minimizers}, any $L^1$-minimizing dual window $h$ is obtained as in \eqref{def: general h} by choosing 
weights $w_k(x)\ge0$ for $k\in J_x$ satisfying
$\sum_{k\in J_x} w_k(x)=1$.
For each $k\in\{-N_\alpha,\dots,N_\alpha\}$ define
$$
E_k:=\{x\in[0,\alpha):\ w_k(x)>0\}.
$$
In particular, since
$\sum_{k\in J_x} w_k(x)=1$ implies at least one $w_k(x)>0$ for a.e. $x$, we have that,  $|\bigcup_{k=-N_\alpha}^{N_\alpha}E_k|=|[0,\alpha)|=\alpha$.
Then, up to sets of measure zero, 
$$\supp(h)=
\bigcup_{k=-N_\alpha}^{N_\alpha}\,\bigl( E_k+\alpha k \bigr),$$
and the sets $E_k+\alpha k$ are pairwise disjoint (since they lie in distinct
intervals $[\alpha k,\alpha(k+1))$). Consequently,
$$
|\supp(h)|
=
\sum_{k=-N_\alpha}^{N_\alpha} |E_k|\geq \Big|\bigcup_{k=-N_\alpha}^{N_\alpha}E_k\Big|=\alpha.
$$
Assume there exists a measurable set $A\subset[0,\alpha)$ with $|A|>0$ and two
distinct indices $k_1\neq k_2$ such that $w_{k_1}(x)>0$ and $w_{k_2}(x)>0$ for
all $x\in A$. 
Then, $A\subset E_{k_1}\cap E_{k_2}$, so every point of $A$ is counted at least twice in the sum $\sum_k |E_k|$. We obtain
\begin{align*}
|\supp(h)|
&= |E_{k_1}|+|E_{k_2}|+\sum_{k\neq k_1,k_2}|E_k|=2|A|+|E_{k_1}\setminus A|+|E_{k_2}\setminus A|+\sum_{k\neq k_1,k_2}|E_k|\\
&\geq \Big|\bigcup_{k=-N_\alpha}^{N_\alpha}E_k\Big|+|A|
=|A|+\alpha
>\alpha.    
\end{align*}

Conversely, suppose that for almost every $x\in[0,\alpha)$ there is exactly
one index $k(x)\in J_x$ such that $w_{k(x)}(x)=1$ and $w_k(x)=0$ for $k\neq 
k(x)$. Then the sets $E_k$ are pairwise disjoint, up to sets of measure zero, and their
union is $[0,\alpha)$. Hence
$$
|\supp(h)|
=
\sum_{k=-N_\alpha}^{N_\alpha}|E_k|
=
\Big|\bigcup_{k=-N_\alpha}^{N_\alpha}E_k\Big|
=
\alpha.$$
\end{proof}

\begin{example}\label{example: cosine}
Let 
$g(x)=\mathbf{1}_{[-1,1]}(x)\cos\!\left(\frac{\pi}{2}x\right)$,
and take $\alpha=1$, $\beta=\frac{1}{2}$. Then $\mathcal{G}(g,1,\frac{1}{2})$ is a Gabor frame in the painless regime. Moreover, $\{|g(x+k)|^2\}$ is a partition of the unity:
\begin{equation*}
\sum_{k\in\mathbb{Z}} |g(x+k)|^2
=
\cos^2\!\left(\tfrac{\pi}{2}x\right)
+
\sin^2\!\left(\tfrac{\pi}{2}x\right)
=1,
\end{equation*}
so the frame operator is a multiplication operator with $ \frac{1}{\beta}G_0^{g,g}(x)=2$.
We now determine an $L^1$-minimizing dual window using Theorem \ref{thm: main thm painless}. For $x\in[0,1]$, the relevant values are
\begin{equation*}
|g(x)|=\cos\!\left(\tfrac{\pi}{2}x\right),
\qquad
|g(x-1)|=\sin\!\left(\tfrac{\pi}{2}x\right).
\end{equation*}
Hence
\begin{equation*}
|g(x)|>|g(x-1)| \quad \text{for } x\in\left[0,\tfrac{1}{2}\right),
\quad \text{
while} \quad
|g(x)|<|g(x-1)| \quad \text{for } x\in\left(\tfrac{1}{2},1\right].
\end{equation*}
Thus, up to sets of measure zero, the maximizing index is unique for every $x\in[0,1]$, and the corresponding selection set is
\begin{equation*}
I=\left(-\tfrac{1}{2},\tfrac{1}{2}\right).
\end{equation*}
Therefore the dual window constructed in Theorem \ref{thm: main thm painless} is
\begin{equation}\label{eq: h example cos}   
h(x)
=
\begin{cases}
\frac{1}{2}\,\sec\!\left(\frac{\pi}{2}x\right), & x\in\left(-\frac{1}{2},\frac{1}{2}\right),\\
0, & \text{otherwise}.
\end{cases}
\end{equation}
Since the maximizing index is unique for almost every $x\in[0,1]$, this $L^1$-minimizing dual window is unique. Furthermore, its $L^1$-norm is
\begin{equation}\label{eq: norm 1 log}
    \|h\|_1
=\frac12\int_{-1/2}^{1/2}
\sec\left(\frac{\pi x}{2}\right) \,dx
=\frac{2}{\pi}\log(1+\sqrt{2}).
%<\frac{2}{\pi}.    
\end{equation}
Finally, note that this frame $\G(g,1,\frac{1}{2})$ is tight, and the canonical dual window is $\widetilde{g}(x)=\frac12\mathbf{1}_{[-1,1]}(x)\cos\!\left(\frac{\pi}{2}x\right)$. Thus, $\widetilde{g}$ does not coincide with the unique $L^1$ minimizer $h$ given by \eqref{eq: h example cos}. Moreover, the  $L^1$-norm of the canonical dual is \begin{equation}\label{eq: can norm 2/pi}
    \|\widetilde g\|_1
=\frac12\int_{-1}^{1}\cos\left(\frac{\pi x}{2}\right)\,dx
=\frac{2}{\pi}.
\end{equation}
Therefore, by simply comparing \eqref{eq: norm 1 log} and \eqref{eq: can norm 2/pi}, since $\frac{2}{\pi}>\frac{2}{\pi}\log(1+\sqrt{2})$, one can verify that, in this example, the canonical dual window does not minimize the $L^1$-norm.

\end{example}

\begin{example}\label{example: non !}
A simple example with non-unique $L^1$-minimizing dual windows is given by $g(x)=\mathbf{1}_{[-1,1]}(x)$, with $\alpha=1$ and $\beta=\frac{1}{2}$.
Then $\mathcal{G}(g,1,\frac{1}{2})$ is a Gabor frame in the painless regime. Indeed, for almost every $x\in\mathbb{R}$,
\begin{equation*}
\sum_{k\in\mathbb{Z}} |g(x+k)|^2 = 2,
\end{equation*}
and therefore $\frac{1}{\beta}G_0^{g,g}(x)= 4$.
For $x\in[0,1]$, the only relevant translates are $x$ and $x-1$, and
\begin{equation*}
|g(x)|=|g(x-1)|=1
\qquad \text{for a.e. } x\in(0,1).
\end{equation*}
Hence, for almost every $x\in(0,1)$, the set of maximizing indices is $J_x=\{0,-1\}$.
Thus the maximizing index is not unique on a set of positive measure, and by Theorem \ref{thm: allL1minimizers} the $L^1$-minimizing dual window is not unique.
More precisely, every $L^1$-minimizing dual window is obtained by choosing a measurable function $w:[0,1]\to[0,1]$
and defining
\begin{equation*}
h(x)=
\begin{cases}
\frac{1}{2}\, w(x), & x\in[0,1],\\ 
\frac{1}{2}\, \bigl(1-w(x+1)\bigr), & x\in[-1,0),\\
0, & \text{otherwise}.
\end{cases}
\end{equation*}
For example, the following are three distinct $L^1$-minimizing dual windows:
\begin{equation*}
h_1(x)=
\begin{cases}
\frac{1}{2}, & x\in[0,1],\\
0, & \text{otherwise};
\end{cases}
\qquad
h_2(x)=
\begin{cases}
\frac{1}{2}, & x\in[-1,0),\\
0, & \text{otherwise};
\end{cases} \qquad 
h_3(x)=
\begin{cases}
\frac14, & x\in[-1,1],\\
0, & \text{otherwise}.
\end{cases}
\end{equation*}
Each of these has $\|h_j\|_1=\frac{1}{2}.$
Finally, among the $L^1$-minimizing dual windows constructed above, the ones with minimal support are $h_1$ and $h_2$. Indeed, $|\supp(h_1)|=1=|\supp(h_2)|$, and  $|\supp(h_3)|=2$.
Since in this example $\alpha=1$, Corollary \ref{coro: main corollary painless support} implies that the minimal possible support measure of an $L^1$-minimizing dual window is $1$. Thus $h_1$ and $h_2$ attain the minimal possible support, whereas $h_3$ does not.
Moreover, we note that the canonical window is $\widetilde{g}=h_3$. Thus, unlike Example \ref{example: cosine}, the canonical dual is an $L^1$-minimizing dual window in this example. However, it is not the unique $L^1$-minimizer and, in this example, it does not have minimal possible support.
\end{example}

\section{Painful Case}\label{sec: painful}

In this section, we investigate dual windows with minimal $L^1$-norm when the frame operator for $\mathcal{G}(g,\alpha,\beta)$ is not a simple multiplication operator. Throughout this section, we will fix the support of $g$ as in the previous section, that is, $\supp(g) \subseteq [-1,1]$, and $\alpha = 1$. It is not immediately clear in this setting that the dual window with minimal $L^1$-norm will have compact support. In fact, it is not even clear that any dual windows have compact support. In Section \ref{sec: discussion}, we discuss an apparent contradiction in the existing literature on this matter.

The main theorem of this section classifies minimal $L^1$-norm dual windows to $\mathcal{G}(g,1,\beta)$ when $\beta = 2/3$. 

By $I_{[-1/2,0]}$ and $I_{[-1,-1/2]}$ we will denote the sets
\begin{equation}\label{eq: int I}
\begin{split}
I_{[-1/2,0]} &= \bigg\{
x\in [-\tfrac{1}{2},0] \; : \;
\frac{2}{3|g(x)|}
\leq
\frac{2}{3|g(x+1)|}
\bigg(1+\left|\frac{g(x-\frac12)}{g(x+\frac12)}\right|\bigg)
\bigg\} \\
I_{[-1,-1/2]} &= \left\{
x\in [-1,-\tfrac{1}{2}]: \, 
\frac{2}{3|g(x+1)|}
\leq
\frac{2}{3|g(x)|}
\left(1+\left|\frac{g(x+\frac32)}{g(x+\frac12)}\right|\right)
\right\}.
\end{split}
\end{equation}
We denote their complements by $J_{[-1/2,0]} = [-\frac{1}{2},0]\setminus I_{[-1/2,0]}$ and $J_{[-1,-1/2]} = [-1,-\frac{1}{2}]\setminus I_{[-1,-1/2]}$. 
For simplicity, throughout this section one can assume that
$g(x)\neq 0$ for a.e. $x\in(-1,1)$,
so that all the quotients appearing in \eqref{eq: int I} are well defined almost everywhere.
On the exceptional null set where a denominator vanishes, one would need to set the window $h$ in \eqref{eq: h pain} below simply as $h=0$.

\begin{theorem}\label{thm: main painful}
Let $g \in L^\infty(\R)$ be a real-valued function with support contained  in $[-1,1]$
%, $g(x)\neq0$ for a.e. $x\in(-1,1)$, 
and such that $\mathcal{G}(g,1,\frac{2}{3})$ is a Gabor frame. Then the function 
\begin{equation}\label{eq: h pain}
    h(x) = \begin{cases}
    \tfrac{2}{3g(x)} & x \in I_{[-1/2,0]} \cup (J_{[-1/2,0]}+1) \cup (I_{[-1,-1/2]}+1) \cup J_{[-1,-1/2]} \\ 
    -\tfrac{2g(x-\frac{5}{2})}{3g(x-\frac{3}{2})g(x-1)} & x\in (J_{[-1/2,0]} + 2)\\ 
    \tfrac{2g(x+\frac{5}{2})}{3g(x+\frac{3}{2})g(x+1)} & x\in (J_{[-1,-1/2]} -1)\\ 
    0 & \text{otherwise}
    \end{cases}
\end{equation}
is a dual window to $\mathcal{G}(g,1,\frac{2}{3})$ with minimal $L^1$-norm. 
\end{theorem}

\smallskip

\subsection{Discussion on the support of the canonical dual window}\label{sec: discussion}

There is a subtle point in the literature concerning compact support of dual windows beyond the painless regime. In \cite{delprete1997rational,del1999estimates}, Del Prete states, in the rational case, that when the frequency step is large enough so that the frame operator is no longer a multiplication operator, the canonical dual does not have compact support. On the other hand, Bölcskei proved a sharp criterion (in the rationally oversampled setting)  \cite{bolcskei1999necessary}: for a compactly supported window, the canonical dual is compactly supported if and only if the associated Zibulski-Zeevi matrix is unimodular. In particular, he gives an explicit example (with $\alpha\beta=\frac{2}{3}$) where the frame operator is not a multiplication operator, yet the canonical dual is compactly supported. Thus, the broad statement \textit{``in the painful case the canonical dual cannot have compact support''}  has to be read with caution. In the appendix, we revise the example provided \cite{bolcskei1999necessary} by using more direct calculation (see Example \ref{example compact supp}).

A related but different perspective is provided by Christensen, Kim, and Kim. In \cite{christensen2010gabor}, for windows $g$ supported on $[-1,1]$, in the normalized setting $\alpha=1$, they prove the following. 

\begin{lemma}\cite[Proof of Theorem 2.1]{christensen2010gabor}\label{lem: Ole}
    Let $\G(g,1,\beta)$ be a Gabor frame with $1/2 < \beta < 1$ and $\supp(g) \subset [-1,1]$. Let ${M} \in \N$ be the integer such that 
    \begin{equation*}
    \tfrac{{M}-1}{{M}} \leq \beta{<} \tfrac{{M}}{{M}+1}.
 \end{equation*}
 If $\mathcal{G}(h,1,\beta)$ is a dual Gabor frame to $\mathcal{G}(g,1,\beta)$, then so is $\mathcal{G}(h\mathbf{1}_J,1,\beta)$ where $J$ is the set 
 \begin{equation}\label{def: J Ole}
 J =\bigg(\bigcup_{k=1}^{{M}-1} \left[-(k+1), -\tfrac{k}{\beta}\right] \bigg) \cup [-1,1] \cup \bigg(\bigcup_{k=1}^{{M}-1} \left[\tfrac{k}{\beta}, k+1\right] \bigg) .     
 \end{equation}
\end{lemma}

From Lemma \ref{lem: Ole}, the following is immediate. 

\begin{theorem}\label{thm: dual support}
    Let $\G(g,1,\beta)$ be a Gabor frame with $1/2 < \beta < 1$ and $\supp(g) \subset [-1,1]$. Let $J$ be as in \eqref{def: J Ole}. Then every dual window to $\G(g,1,\beta)$ with minimal $L^p$-norm is supported on $J$ for $1 \leq p < \infty$. In particular, the canonical dual window to $\G(g,1,\beta)$ is supported on $J$. 
\end{theorem}

More generally, assuming the existence of one compactly supported dual window, Stoeva gives a characterization of all compactly supported dual windows of a Gabor frame; see \cite[Theorem 1.1]{stoeva2022compactly}. Also, \cite[Proposition 1.2]{stoeva2022compactly} provides an iterative procedure in which every iterate is an exact dual window and, when the original window and the initial dual window are compactly supported, every iterate remains compactly supported and converges to the canonical dual window.  

Complementary results in the normalized setting $\alpha=1$ and $\supp(g)=[-1,1]$ were obtained by Lemvig and Nielsen. In particular, \cite[Lemma 3.3]{lemvig2019gabor} gives an explicit parametrization of all dual windows having sufficiently small support. This parametrization also yields characterizations of boundedness, continuity, and smoothness of the resulting dual windows. In particular, the authors show that a compactly supported dual window can be chosen with the same smoothness as the original window and that, in general, this degree of smoothness is optimal.

\subsection{The Case of $\beta = 2/3$}\label{sec: painful 2/3} 
To derive the minimal dual window, we will show that such a dual is necessarily the solution (pointwise) to a constrained $\ell^1$ minimization problem of the following form.

\begin{lemma}\label{lem: painful case min problem}
    Let $a,b,c,d\in \mathbb R\setminus\{0\}$ 
    be constants, and consider the optimization problem
    \begin{align*}
\min_{\R^3}\; &|x_1| + |x_2| + |x_3| \\
 \mathrm{subject \, to\ }\; & a x_1 + bx_2 = \tfrac{2}{3} \quad \mathrm{and} \quad cx_2 + dx_3 = 0.
\end{align*}
The solution occurs at the point $(\frac{2}{3a},0,0)$ if $\frac{2}{3|a|} \leq \frac{2}{3|b|}(1+\left|\frac{c}{d}\right|)$, or at the point $(0, \frac{2}{3b}, -\frac{2c}{3db})$ if $\frac{2}{3|a|} \geq \frac{2}{3|b|}(1+\left|\frac{c}{d}\right|)$. If equality holds, then every point on the line segment joining these two points is also a minimizer.
\end{lemma}

We defer the proof of Lemma \ref{lem: painful case min problem} to the Appendix \ref{sec: appendix}, as it relies only on elementary optimization arguments and including it here would interrupt the flow of the exposition.

\begin{proof}[Proof of Theorem \ref{thm: main painful}] A window $h$ dual to $\mathcal{G}(g,1,2/3)$ for $\supp(g) \subseteq [-1,1]$ must satisfy the equations 
\begin{equation}\label{eq: painful-duality} \begin{aligned} 
g(x)h(x)+g(x+1)h(x+1) &=\tfrac{2}{3}, &&x\in[-1,0],\\
g(x-\tfrac12)h(x+1)+g(x+\tfrac12)h(x+2) &=0, &&x\in[-\tfrac12,0],\\ 
g(x+\tfrac12)h(x-1)+g(x+\tfrac32)h(x) &=0, &&x\in[-1,-\tfrac12]. 
\end{aligned} 
\end{equation}
Additionally, by Theorem \ref{thm: dual support} we may assume that $\supp(h) \subseteq A\cup B$ where
\begin{align*}
    A = \left[-\tfrac{1}{2},0\right]\cup \left[\tfrac{1}{2},1\right] \cup \left[\tfrac{3}{2},2\right]; \qquad 
    B = \left[-2,-\tfrac{3}{2}\right] \cup \left[-1,-\tfrac{1}{2}\right] \cup \left[0,\tfrac{1}{2}\right].
\end{align*}

Now, we write
\begin{align*}
    \|h\|_1 &= \int_\R |h(x)|dx \\
            &= \int_{A} |h(x)|dx + \int_{B} |h(x)|dx \\
            &= \int_{[-\frac{1}{2},0]} |h(x)| + |h(x+1)| + |h(x+2)|dx \\
            &+ \int_{[-1,-\frac{1}{2}]} |h(x-1)| + |h(x)| + |h(x+1)|dx.
\end{align*}
Thus, the minimal dual solves the following pointwise optimization problems: 

\noindent for all $x \in [-\frac{1}{2},0]$, solves
\begin{equation}\label{eq: painful-pointwise-A}
\begin{aligned}
\min_{\R^3}\; |h(x)| + |h(x+1)| + |h(x+2)| \\
 \text{subject to }\;  g(x)h(x) + g(x+1)h(x+1) &= \tfrac{2}{3}\\
              g(x-\tfrac{1}{2})h(x+1) + g(x+\tfrac{1}{2})h(x+2) &= 0,
\end{aligned}
\end{equation}
and for all $x \in [-1,-\frac{1}{2}]$, solves 
\begin{equation}\label{eq: painful-pointwise-B}
\begin{aligned}
\min_{\R^3}\; |h(x-1)| + |h(x)| + |h(x+1)| \\
 \text{subject to }\;  g(x)h(x) + g(x+1)h(x+1) &= \tfrac{2}{3}\\
              g(x+\tfrac{1}{2})h(x-1) + g(x+\tfrac{3}{2})h(x) &= 0.
\end{aligned}
\end{equation}

By Lemma \ref{lem: painful case min problem}, it follows that for $x \in [-1/2,0]$, if 
\begin{equation*}
\frac{2}{3|g(x)|} \leq \frac{2}{3|g(x+1)|}\bigg(1+ \left|\frac{g(x-\frac{1}{2})}{g(x+\frac{1}{2})}\right|\bigg),
\end{equation*}
then 
\begin{equation*}
h(x)= \frac{2}{3g(x)}, \quad h(x+1) =h(x+2) = 0.
\end{equation*}
Otherwise, we have 
\begin{equation*}
h(x) = 0, \; h(x+1) = \frac{2}{3g(x+1)}, \quad h(x+2) = -\frac{g(x-\frac{1}{2})}{g(x+\frac{1}{2})}\frac{2}{3g(x+1)}. 
\end{equation*} 
In the notation of the Theorem, this says that 
\begin{align*}
h|_{I_{[-1/2,0]}\cup (J_{[-1/2,0]}+1)} &= \frac{2}{3g(x)} \\ h|_{(J_{[-1/2,0]}+2)} &= -\frac{2g(x-\frac{5}{2})}{3g(x-\frac{3}{2})g(x-1)}.
\end{align*}

Similarly, for $x \in[-1,-1/2]$, we have if
\begin{equation*}
\frac{2}{3|g(x+1)|} \leq \frac{2}{3|g(x)|}\bigg(1
+ \left|\frac{g(x+\frac{3}{2})}{g(x+\frac{1}{2})}\right|\bigg),
\end{equation*}
then 
\begin{equation*}
h(x+1) = \frac{2}{3g(x+1)}, \quad h(x) = h(x-1) = 0.
\end{equation*}
Otherwise, 
\begin{equation*}
h(x+1) = 0,\; h(x) = \frac{2}{3g(x)}, \quad h(x-1) = \frac{g(x+\frac{3}{2})}{g(x+\frac{1}{2})}\frac{2}{3g(x)}.
\end{equation*}
Again, in the notation of the theorem, this says 
\begin{align*}
h|_{(I_{[-1,-1/2]}+1)\cup J_{[-1,-1/2]}} &= \frac{2}{3g(x)} \\
h|_{(J_{[-1,-1/2]}-1)} &= \dfrac{2g(x+\frac{5}{2})}{3g(x+\frac{3}{2})g(x+1)}.
\end{align*}

It remains to verify that $\mathcal G(h,1,\frac{2}{3})$ is a Bessel sequence. To that end, we will compare $h$ with the canonical dual window $\widetilde g$. By Theorem \ref{thm: dual support}, $\widetilde g$ is supported on $A\cup B$. Moreover, since $\mathcal G(\widetilde g,1,\frac{2}{3})$ is a Bessel sequence, the upper Gabor-frame estimate implies that $\widetilde g\in L^\infty(\mathbb R)$.
For each $x\in[-\frac{1}{2},0]$, the vector
\begin{equation*}
\bigl(\widetilde g(x),\widetilde g(x+1),\widetilde g(x+2)\bigr)
\end{equation*}
is feasible for the same finite-dimensional optimization problem \eqref{eq: painful-duality} as
\begin{equation*}
\bigl(h(x),h(x+1),h(x+2)\bigr).
\end{equation*}
Therefore, by the pointwise $\ell^1$-minimality of $h$ \eqref{eq: painful-pointwise-A},
\begin{equation*}
|h(x)|+|h(x+1)|+|h(x+2)|
\leq
|\widetilde g(x)|+|\widetilde g(x+1)|+|\widetilde g(x+2)|
\leq
3\|\widetilde g\|_\infty.
\end{equation*}
Similarly, for $x\in[-1,-\frac{1}{2}]$, looking now at the pointwise $\ell^1$-minimality of $h$ \eqref{eq: painful-pointwise-B}, 
\begin{equation*}
|h(x-1)|+|h(x)|+|h(x+1)|
\leq
3\|\widetilde g\|_\infty.
\end{equation*}
Since these two collections of translates cover the support set $A\cup B$, it follows that
\begin{equation*}
\|h\|_\infty\leq 3\|\widetilde g\|_\infty.
\end{equation*}
Thus, $h$ is bounded and compactly supported (and hence $h\in W(\mathbb R)$).
Therefore, $\mathcal G(h,1,\frac{2}{3})$ is a Bessel sequence, which completes the proof. 

\end{proof}

\begin{example}
    \emph{Let $g(x) = \sqrt{|x|}\mathbf{1}_{[-1,1]}(x)$. Set $a_1 = -\frac{3}{4}-\frac{\sqrt{5}}{20}\sqrt{16\sqrt{5}-35}$. and $a_2 = -\frac{1}{4}+\frac{\sqrt{5}}{20}\sqrt{16\sqrt{5}-35}$. Then, if $\alpha = 1$, $\beta = \frac{2}{3}$, the dual window with minimal $L^1$-norm is:} 
    \begin{equation*}h(x) = \begin{cases}
        \frac{\beta}{g(x)} & x \in [-1,a_1]\cup[-\frac{1}{2},a_2] \cup [1+a_1,\frac{1}{2}] \cup  [1+a_2,1]\\
        \frac{2}{3}\sqrt{\frac{5/2 - x}{(x-3/2)(x-1)}} & x \in [2+a_2, 2] \\
        \frac{2}{3}\sqrt{\frac{x+5/2}{(-x-3/2)(-x-1)}} & x \in [-2,a_1-1].
    \end{cases} 
    \end{equation*}
\end{example}

\smallskip

\subsection{The General Case}\label{sec: painful case}
We conclude with an observation about the structure of the problem in the general case, that is, for a general $\beta$ in the painful regime.

Let $\G(g,\alpha,\beta)$ be a Gabor frame with $\supp(g)\subseteq[-1,1]$, $\alpha=1$, and assume
$
\frac{M-1}{M}\leq \beta < \frac{M}{M+1}$,
$M\geq 2$.
By Lemma \ref{lem: Ole}, an $L^1$-minimizing dual window $h$
may be chosen supported on the associated truncation set $J$ (see \eqref{def: J Ole}).
Since $\alpha=1$, the Ron--Shen/Janssen equations
\eqref{eq: characterization dual} are $\mathbb Z$--periodic in $x$,
so it suffices to verify them on a fundamental domain,
which we take to be $[-1,0)$.
Because $J$ is bounded, for each fixed $x\in[-1,0)$ only finitely
many integer shifts $h(x+k)$ can be nonzero. Consequently, the
duality equations \eqref{eq: characterization dual} involve only finitely many unknowns. 
Therefore, the
minimization of $\|h\|_1$ reduces almost everywhere to a
finite-dimensional $\ell^1$ optimization problem for the vector
$$
(h(x+k_1),\dots,h(x+k_r)),
$$
that is, the samples of $h$ along the lattice $x+\mathbb Z$,
where $\{k_1,\dots,k_r\}=\{k\in\mathbb Z:\,x+k\in J\}$, for some $r=r(x)$ denoting the number of integer shifts whose
translates intersect $J$. 
Precisely, because $J$ is a bounded finite union of intervals, the set
$\{k\in\mathbb Z:\ x+k\in J\}$
changes only when $x$ crosses an endpoint of one of the translated intervals $J-k$. Hence the counting function
$
r(x)=\#\{k\in\mathbb Z:\ x+k\in J\}
$
is piecewise constant on $[-1,0)$, with possible jumps only at finitely many points. Therefore, up to a finite exceptional set, we may assume that $r(x)\equiv r$ is constant.

In particular, when $\beta=\frac{M-1}{M}$, we have that $J$ in \eqref{def: J Ole} becomes
$$
J=\left(\bigcup_{k=1}^{M-1}\left[-(k+1),-\tfrac{kM}{M-1}\right]\right)\cup[-1,1]\cup\left(\bigcup_{k=1}^{M-1}\left[\tfrac{kM}{M-1},k+1\right]\right).
$$
One has $r(x)=M$ for all
$x\in[-1,0)$ except possibly finitely many values of $x$. Indeed,
for $x\in[-1,0)$, the indices $k=0$ and $k=1$ always satisfy
$
x,\ x+1\in[-1,1]\subset J
$.
On the other hand, for each $k=1,\dots,M-2$, the two integers $-k$ and $M-k$ satisfy
$$
x-k\in \left[-(k+1),-\tfrac{kM}{M-1}\right]
\quad\Longleftrightarrow\quad
x\le -\tfrac{k}{M-1},
$$
and
$$
x+(M-k)\in \left[\tfrac{(M-1-k)M}{M-1},\,M-k\right]
\quad\Longleftrightarrow\quad
x\ge -\tfrac{k}{M-1}.
$$
Hence, away from the finitely many boundary points
($
\left\{-\frac{k}{M-1}:k=1,\dots,M-2\right\}
$),
exactly one of the two integers
$-k$ or $M-k$
contributes to the set
$
\{j\in\mathbb Z:\ x+j\in J\}
$.
Therefore,
$$
r(x)=2+(M-2)=M
$$
for all $x\in[-1,0)$ except possibly finitely many values. Thus, for almost every
$x$, the pointwise problem involves $M$ unknowns.
In addition, for each fixed $x\in[-1,0)$, only the equations in
\eqref{eq: characterization dual} corresponding to $n=0$ and to at most
$M-2$ additional values of $n$ can be nontrivial. Hence, for almost every
$x$, the pointwise duality conditions define a finite linear system with at
most $M-1$ equations in the $M$ unknowns
$$
u(x)=\big(h(x+k_1),\dots,h(x+k_M)\big).
$$
Therefore, for almost every $x$, one has to solve an optimization problem of
the form
$$
\min_{u\in\mathbb R^M}\|u\|_{\ell^1}
\qquad\text{subject to}\qquad
A(x)u=(\beta,0,\dots,0)^{\top},
$$
where the matrix $A(x)$ is determined by the non-vanishing samples of $g$ appearing in
\eqref{eq: characterization dual}. Thus, the search for an $L^1$-minimizing
dual window reduces pointwise to a family of finite-dimensional
$\ell^1$-minimization problems. The case $\beta=\frac23$ ($M=3$), treated
above, corresponds to minimizing $\|u\|_{\ell^1}$ in $\mathbb R^3$ subject
to two linear constraints.

\section{Gabor Frames with Diagonally Dominant Frame Operator}\label{sec: diag dominant}

Here we consider windows $g\in L^2(\R)$  no longer compactly supported but generating a Gabor frame $\G(g,\alpha,\beta)$ with diagonally dominant frame operator, for example, $g(x)=e^{-\pi x^2}$ with small $\beta$. If $h$ is a dual window to $\G(g,\alpha, \beta)$, then from \eqref{eq: characterization dual}
\begin{equation}\label{eq: = beta}
    \sum_{k\in\Z} {g(x+ k\alpha)}h(x+ k\alpha)=\beta \qquad \text{ for a.e. } x\in[0,\alpha].    
\end{equation}
Similarly as in \eqref{eq: > beta}, this implies 
\begin{equation}
   \sum_{k\in\Z}|h(x+ k\alpha)|\geq \frac{\beta}{\|g\|_\infty} \qquad \text{ for a.e. } x\in[0,\alpha],     
\end{equation}
from which it follows that
\begin{equation}\label{eq: bound approx case}
    \|h\|_1\geq\frac{\alpha\beta}{\|g\|_\infty}.
\end{equation}

In general, as defined in  \cite{christensen2010approximately}, 
an \textit{approximately dual frame} $\{h_k\}$ associated with a given frame $\{g_k\}$ is defined by satisfying
\begin{equation}\label{eq: approx dual}
\left\| f - \sum \langle f, g_k \rangle h_k \right\|_2
\leq \varepsilon \|f\|_2, \qquad \forall f \in L^2(\R),
\end{equation}
for some $\varepsilon < 1$, which controls the distortion in the frame duality condition \eqref{eq:dual_frame}.

In the case of Gabor systems $\G(g,\alpha,\beta)$, $\G(h,\alpha, \beta)$, assuming that the Walnut representation \eqref{eq: Walnut} holds (for instance, if the windows are in the Wiener amalgam space $g,h\in W(\mathbb R)$), the condition \eqref{eq: approx dual} reads as
\begin{equation}\label{eq: approx Gabor dual frame} 
    \left\|f-\frac{1}{\beta}\sum_{n\in \Z}G_n^{g,h} \, \cdot (T_{-n/\beta}f)\right\|_2\leq \varepsilon\|f\|_2 \qquad \forall f \in L^2(\R).
\end{equation}
If we impose \eqref{eq: = beta} to be satisfied exactly, i.e., $G_0^{g,h}=\beta$, then \eqref{eq: approx Gabor dual frame} reduces to 
\begin{equation}
    \left\|\frac{1}{\beta}\sum_{n\neq 0}G_n^{g,h} \, \cdot (T_{-n/\beta}f)\right\|_2\leq \varepsilon\|f\|_2 \qquad \forall f \in L^2(\R).
\end{equation}

Observe that for each $n\neq 0$, 
\begin{equation*}
    \left\|G_n^{g,h} \, \cdot (T_{-n/\beta}f)\right\|_2\leq \|G_n^{g,h}\|_\infty \|f\|_2.
\end{equation*}
So, to get the approximate-duality estimate \eqref{eq: approx Gabor dual frame} one can aim to obtain 
\begin{equation}\label{eq: eps relation}
    \frac{1}{\beta}\sum_{n\neq 0}\|G_n^{g,h}\|_\infty<\varepsilon.
\end{equation}
We state this observation in the following lemma:

\begin{lemma}\label{lem: crit for approx dual window}
    If $g,h \in L^2(\R)$ are such that the Walnut representation \eqref{eq: Walnut} for $S_{g,h}$ holds, and
    \begin{align*}
   G_0^{g,h}\equiv \beta, \qquad \|G_n^{g,h}\|_\infty\leq\epsilon_n \quad \text{ for } n\neq 0, \quad \text{ with } \quad \frac{1}{\beta}\sum\limits_{n\in \Z\setminus\{0\}} \epsilon_n < 1,
    \end{align*}
    then $h$ is an approximate dual window to $\G(g,\alpha,\beta)$.
\end{lemma}
For the rest of this section, we will assume that our window function $g$ lies in the class of window functions $\mathcal{V}$ used in \cite{christensen2018approximately}. 

\begin{definition} The class $\mathcal{V}$ consists of all continuous $g\in L^2(\R)$ such that $g$ is positive valued, even, decreasing on $[0,\infty)$, and  
\begin{equation*}
g(x) \leq \frac{c}{1+|x|^{\sigma+1}}
\end{equation*}
for some constants $c,\sigma > 0$.
\end{definition}

\begin{lemma}\label{lem: g in V}
Let $g\in\mathcal{V}$, and fix $\alpha,r,H>0$ and
$0<\epsilon<1$. Then there exists $
\beta_0=\beta_0(g,\alpha,r,H,\epsilon)>0
$
such that, for every $0<\beta<\beta_0$ and every
$h\in L^2(\R)$ satisfying
\begin{equation}\label{eq: support and bound cond}
\supp(h)\subseteq[-r,r] \qquad \text{ and } \qquad
\|h\|_\infty\leq H,    
\end{equation}
we have
$$
\frac{1}{\beta}
\sum_{n\neq 0}\|G_n^{g,h}\|_\infty
\leq\epsilon.
$$
In particular, if
$G_0^{g,h}(x)\equiv\beta,
$
then $h$ is an approximate dual window to
$\G(g,\alpha,\beta)$.
\end{lemma}

\begin{proof}[Proof of Lemma \ref{lem: g in V}]
Let $h\in L^2(\R)$ be any function satisfying \eqref{eq: support and bound cond}.
Set
$$
N_{r,\alpha}
:=
1+\Big\lceil \tfrac{2r}{\alpha}\Big\rceil.
$$
Fix $n\neq 0$. By definition,
$$
G_n^{g,h}(x)
=
\sum_{k\in\Z}
g\left(x+\frac{n}{\beta}+\alpha k\right)
h(x+\alpha k).
$$
Since $\supp(h)\subseteq[-r,r]$, the summand can be nonzero only if
$x+\alpha k\in[-r,r]$, i.e.,
$$
\frac{-r-x}{\alpha}
\leq k
\leq
\frac{r-x}{\alpha}.
$$
Hence, for each fixed $x$, there are at most $N_{r,\alpha}$
nonzero terms in the sum. Therefore,
$$
|G_n^{g,h}(x)|
\leq
H N_{r,\alpha}
\sup_{y\in[-r,r]}
\left|g\left(y+\frac{n}{\beta}\right)\right|.
$$
Taking the supremum over $x$ yields
$$
\|G_n^{g,h}\|_\infty
\leq
H N_{r,\alpha}
\sup_{y\in[-r,r]}
\left|g\left(y+\frac{n}{\beta}\right)\right|.
$$

Now assume in addition that
$$
0<\beta
\leq
\min\left\{1,\frac{1}{2r}\right\}.
$$
Then, for every $n\neq0$, we have
$$
\frac{|n|}{\beta}-r
\geq
\frac{|n|}{2\beta},
$$
and since $g\in\mathcal V$, there exist constants $c>0$ and
$\sigma>0$ such that
$$
|g(t)|
\leq
\frac{c}{1+|t|^{1+\sigma}},
\qquad t\in\R.
$$
Because $g$ is even and decreasing on $[0,\infty)$,
\begin{align*}
\sup_{y\in[-r,r]}
\left|g\left(y+\frac{n}{\beta}\right)\right|
&=
\begin{cases}
g\left(\frac{n}{\beta}-r\right),
& n>0,\\ 
g\left(\frac{n}{\beta}+r\right),
& n<0,
\end{cases}\\
&\leq
\frac{c}{
1+\left(\frac{|n|}{\beta}-r\right)^{1+\sigma}
}\\
&\leq
\frac{c(2\beta)^{1+\sigma}}{|n|^{1+\sigma}}.
\end{align*}
Consequently,
$$
\frac{1}{\beta}
\|G_n^{g,h}\|_\infty
\leq
H N_{r,\alpha}c\,2^{1+\sigma}\beta^\sigma
\frac{1}{|n|^{1+\sigma}}.
$$
Summing over $n\neq0$, we obtain
\begin{equation}\label{eq: aux series}
\frac{1}{\beta}
\sum_{n\neq0}\|G_n^{g,h}\|_\infty
\leq
H N_{r,\alpha}c\,2^{1+\sigma}\beta^\sigma
\sum_{n\neq0}\frac{1}{|n|^{1+\sigma}}.
\end{equation}
Since $\sigma>0$, the series
$$
\sum_{n\neq0}\frac{1}{|n|^{1+\sigma}}
$$
converges, and the right-hand side of
\eqref{eq: aux series} tends to $0$ as $\beta\to0^+$.
Choose $\beta_0>0$ such that
$$
\beta_0
\leq
\min\left\{
1,
\frac{1}{2r},
\left(
\tfrac{\epsilon}{
H N_{r,\alpha}c\,2^{1+\sigma}
\displaystyle\sum_{n\neq0}\tfrac{1}{|n|^{1+\sigma}}
}
\right)^{1/\sigma}
\right\}.
$$
Then, for every $0<\beta<\beta_0$,
$$
H N_{r,\alpha}c\,2^{1+\sigma}\beta^\sigma
\sum_{n\neq0}\frac{1}{|n|^{1+\sigma}}
\leq\epsilon.
$$
Since this choice depends on $h$ only through the
fixed constants $r$ and $H$, the same $\beta_0$ works uniformly
for every $h$ satisfying the stated support and boundedness
conditions \eqref{eq: support and bound cond}.

Moreover, the decay assumption on $g$ implies that
$g\in W(\R)$, while every bounded, compactly supported $h$ belongs
to $W(\R)$. Hence the corresponding Gabor systems are Bessel
sequences and the Walnut representation is valid. Therefore, if
$G_0^{g,h}\equiv\beta$, the approximate-duality criterion stated in Lemma \ref{lem: crit for approx dual window} implies that $h$ is an approximate dual window to
$\G(g,\alpha,\beta)$.
\end{proof}

The following theorem demonstrates that for $g\in\mathcal{V}$, we can
find approximate dual windows to $\G(g,\alpha,\beta)$ with nearly
minimal $L^1$-norms.

\begin{theorem}\label{thm: diag dom}
Let $g\in\mathcal{V}$.
For every $\alpha>0$, there exists 
$\beta_0=\beta_0(\alpha,g)>0$
such that, for every $0<\beta<\beta_0$, the function
\begin{equation}\label{eq: h alpha beta}
   h_{\alpha,\beta}(x)
:=
\frac{\beta}{g(x)}
\mathbf{1}_{[-\frac{\alpha}{2},\frac{\alpha}{2}]}(x) 
\end{equation}
is an approximate dual window to $\G(g,\alpha,\beta)$. Moreover,
$$
\frac{
\|{h_{\alpha,\beta}}\|_1
-
\alpha\beta/\|g\|_\infty
}{
\alpha\beta/\|g\|_\infty
}
\longrightarrow0
\qquad\text{as }\alpha\to0,
$$
and such convergence is independent
of $\beta$ (in particular, it holds for any choice
$\beta=\beta(\alpha)$ satisfying $0<\beta(\alpha)<\beta_0(\alpha,g)$).
\end{theorem}

\begin{proof}[Proof of Theorem \ref{thm: diag dom}]
Fix $\alpha>0$, and set
$$
m_\alpha
:=
\min_{|x|\leq\alpha/2}g(x).
$$
Since $g$ is continuous and strictly positive, we
have $m_\alpha>0$. For every $0<\beta\leq1$, the function
$h_{\alpha,\beta}$ satisfies
$$
\supp(h_{\alpha,\beta})
\subseteq
\left[-\tfrac{\alpha}{2},\tfrac{\alpha}{2}\right]
\qquad \text{ and }\qquad 
\|h_{\alpha,\beta}\|_\infty
\leq
\tfrac{\beta}{m_\alpha}
\leq
\tfrac{1}{m_\alpha}.
$$
Thus, Lemma \ref{lem: g in V}, applied with
$r=\tfrac{\alpha}{2}$, $H=\tfrac{1}{m_\alpha}$ (and a fixed $0<\epsilon<1$ for instance $\epsilon=\frac{1}{2}$),
yields a number $\beta_0=\beta_0(\alpha,g)>0$ such that the conclusion of
the lemma holds for every $0<\beta<\beta_0$.

In addition, for a.e. $x\in[0,\alpha)$, there is exactly one
$k\in\mathbb{Z}$ such that $x+\alpha k
\in
\left[-\frac{\alpha}{2},\frac{\alpha}{2}\right]$,
and hence
$$
G_0^{g,{h_{\alpha,\beta}}}(x)
=
\sum_{k\in\mathbb{Z}}
g(x+\alpha k)
{h_{\alpha,\beta}}(x+\alpha k)
=
\beta.
$$
Lemma \ref{lem: g in V} therefore implies that
$h_{\alpha,\beta}$ is an approximate dual window to
$\G(g,\alpha,\beta)$.

To prove the second statement, observe that $g$ is positive, even,
and decreasing on $[0,\infty)$, so $\|g\|_\infty=g(0)$.
Moreover,
$$
\|{h_{\alpha,\beta}}\|_1
=
\beta
\int_{-\alpha/2}^{\alpha/2}
\frac{1}{g(x)}\,dx.
$$
Consequently,
$$
{
\frac{
\|h_{\alpha,\beta}\|_1
}{
\alpha\beta/\|g\|_\infty
}
=
\frac{\|g\|_\infty}{\alpha}
\int_{-\alpha/2}^{\alpha/2}
\frac{1}{g(x)}\,dx
}
\longrightarrow
\frac{\|g\|_\infty}{g(0)}
=
1
$$
as $\alpha\to0$. Therefore,
$$
\frac{
\|h_{\alpha,\beta}\|_1
-
\alpha\beta/\|g\|_\infty
}{
\alpha\beta/\|g\|_\infty
}
\longrightarrow0
\qquad\text{as }\alpha\to0.
$$
Since the expression on the left-hand side is
independent of $\beta$, the convergence holds uniformly with respect
to $\beta>0$, and in particular along any admissible choice
$0<\beta(\alpha)<\beta_0(\alpha,g)$.
\end{proof}

\begin{remark}
Equation \eqref{eq: bound approx case} shows that
$
{\alpha\beta}/{\|g\|_\infty}
$
is a natural lower bound for the $L^1$-norm of any dual window $h$ satisfying
\eqref{eq: = beta}. Theorem \ref{thm: diag dom} shows that the approximate dual
windows $h_{\alpha,\beta}$, defined by \eqref{eq: h alpha beta}, asymptotically attain this bound, in the sense that their
relative error with respect to $\alpha\beta/\|g\|_\infty$ tends to zero as
$\alpha\to 0$.
\end{remark}

\begin{example}[Gaussian window]
    In the case $g(x)=e^{-\pi x^2}$ ($\|g\|_\infty$=1), 
    let $h_{\alpha,\beta}$ as in \eqref{eq: h alpha beta}:
    $$h_{\alpha,\beta}(x)=\frac{\beta}{g(x)}\mathbf{1}_{[-\frac{\alpha}{2},\frac{\alpha}{2}]}(x)=\beta e^{\pi x^2}\mathbf{1}_{[-\frac{\alpha}{2},\frac{\alpha}{2}]}(x).$$
    As above, we have
    \begin{equation*}
        \left|\frac{\|h_{\alpha,\beta}\|_1-\alpha\beta}{\alpha\beta}\right|\longrightarrow 0 \qquad \text{as } \alpha\to 0.
    \end{equation*}
    Moreover, $\G(h_{\alpha,\beta},\alpha,\beta)$ defines an approximately dual frame for $G(e^{-\pi x^2},\alpha,\beta)$ for sufficiently small $\beta$. Indeed, this can be checked more explicitly than the general estimates of Lemma \ref{lem: g in V}.
    
    For a.e. $x\in [0,\alpha)$, there is exactly one $k_x\in \Z$ such that $x+\alpha k_x\in [-\frac{\alpha}{2},\frac{\alpha}{2}]$. So, for $n\neq 0$,
    \begin{align*}
        G_n^{g,h_{\alpha,\beta}}(x)&=\beta\sum_{k\in \Z}e^{-\pi (x+\alpha k +\frac{n}{\beta})^2}e^{\pi (x+\alpha k)^2}\mathbf{1}_{[-\frac{\alpha}{2},\frac{\alpha}{2}]}(x+\alpha k) \\&=\beta e^{-\pi\left(\frac{n^2}{\beta^2}+2(x+k_x\alpha)\frac{n}{\beta}\right)}\\
        &\leq \beta e^{-\pi\left(\frac{n^2}{\beta^2}-\frac{|n|\alpha}{\beta}\right)}
    \end{align*}
    Thus, in the view of \eqref{eq: eps relation}, we get that for $0\leq\beta\leq \frac{1}{2\alpha}$ 
    \begin{equation}\label{eq: bound G_n h_alpha}
         \frac{1}{\beta}\sum_{n\neq 0}\|G_n^{g,h_{\alpha,\beta}}\|_\infty=\sum_{n\neq 0}e^{-\pi\left(\frac{n^2}{\beta^2}-\frac{|n|\alpha}{\beta}\right)}\leq 2\sum_{n=1}^\infty e^{-\frac{\pi n}{2\beta^2}}=\frac{2e^{-\frac{\pi}{2\beta^2}}}{1-e^{-\frac{\pi}{2\beta^2}}}.
    \end{equation}
    It can be checked that the RHS in \eqref{eq: bound G_n h_alpha} is less than 1 so long as $\beta < \sqrt{\frac{\pi}{\log{9}}}$.
        Hence, for all $\beta < \min\{\frac{1}{2\alpha}, \sqrt{\frac{\pi}{\log{9}}}\}$, $\G(\beta e^{\pi x^2}\mathbf{1}_{[-\frac{\alpha}{2},\frac{\alpha}{2}]}(x), \alpha, \beta)$ is an approximate dual frame to $\G(e^{-\pi x^2}, \alpha, \beta)$. 
\end{example}

\section{Appendix}\label{sec: appendix}

The following is the proof of the optimization auxiliary lemma utilized in Section \ref{sec: painful 2/3}.

\begin{proof}[Proof of Lemma \ref{lem: painful case min problem}]
For simplicity assume $a,b,c,d>0$. From the second constraint we obtain
\begin{equation*}
x_3=-\tfrac{c}{d}x_2,
\end{equation*}
and substituting into the first constraint gives
\begin{equation*}
x_1=\frac{\frac23-bx_2}{a}.
\end{equation*}
Therefore the minimization problem reduces to the one-variable problem
\begin{equation*}
\min_{x_2\in\R} F(x_2),
\qquad
F(x_2):=\frac{1}{a}\left|\frac23-bx_2\right|
+\left(1+\frac{c}{d}\right)|x_2|.
\end{equation*}

Since $F$ is convex and piecewise linear, its minimum is attained at a breakpoint or on an interval where the slope is zero. The only breakpoints are
\begin{equation*}
x_2=0
\qquad\text{and}\qquad
x_2=\tfrac{2}{3b}.
\end{equation*}
Thus it is enough to compare the values of $F$ at these two points.

If $x_2=0$, then $x_1=\tfrac{2}{3a}$, $x_3=0$, $
F(0)=\tfrac{2}{3a}$.

If $x_2=\frac{2}{3b}$, then
$x_1=0$, $x_3=-\tfrac{c}{d}\tfrac{2}{3b}=-\tfrac{2c}{3bd}$,
so $F\!\left(\tfrac{2}{3b}\right)
=
\tfrac{2}{3b}\left(1+\tfrac{c}{d}\right)$.

Hence,
\begin{equation*}
F(0)<F\!\left(\tfrac{2}{3b}\right)
\quad\Longleftrightarrow\quad
\tfrac{2}{3a}<\tfrac{2}{3b}\left(1+\tfrac{c}{d}\right),
\end{equation*}
in which case the unique minimizer is $\left(\tfrac{2}{3a},0,0\right)$; and
\begin{equation*}
F(0)>F\!\left(\tfrac{2}{3b}\right)
\quad\Longleftrightarrow\quad
\tfrac{2}{3a}>\tfrac{2}{3b}\left(1+\tfrac{c}{d}\right),
\end{equation*}
in which case the unique minimizer is $\left(0,\frac{2}{3b},-\frac{2c}{3bd}\right)$.

Finally, if
\begin{equation*}
\tfrac{2}{3a}=\tfrac{2}{3b}\left(1+\tfrac{c}{d}\right),
\end{equation*}
then the two endpoint values coincide. Since $F$ is convex and affine on the interval
\begin{equation*}
\left[0,\tfrac{2}{3b}\right],
\end{equation*}
it follows that every point on the corresponding line segment between
\begin{equation*}
\left(\tfrac{2}{3a},0,0\right)
\quad\text{and}\quad
\left(0,\tfrac{2}{3b},-\tfrac{2c}{3bd}\right)
\end{equation*}
is also a minimizer.
\end{proof}

The following is a minor adaptation of the example provided at the end  of \cite{bolcskei1999necessary}, which we revise here to justify that,  in the painful case, the canonical dual could have compact support (although examples like this one are ``rare'' as mentioned in \cite{bolcskei1999necessary} at the end of Sec. 2).

\begin{example}\label{example compact supp}
Consider the Gabor system with
\begin{equation*}
\alpha=1,\qquad \beta=\frac23,\qquad g(x)=\mathbf{1}_{[-1,1]}(x).
\end{equation*}
This is a painful case, since $\alpha\beta=\frac23>\frac{1}{2}$, so the frame operator is no longer a multiplication operator. Nevertheless, the canonical dual window can be computed explicitly and we will see that it is still compactly supported.
Using the Walnut representation, the canonical dual $\widetilde{g}$ is determined by the equation
$S_{g,g}\widetilde{g}=g$.
For this particular rectangle window,  only
$n=0,\pm1$ contribute a.e. in \eqref{eq: Walnut}, because shifts by $3$ or more do not overlap.
One can compute that, for a.e. $x$,
\begin{equation*}
G_0^{g,g}(x)=2, \qquad
G_1^{g,g}(x)=
\begin{cases}
1, & x\in [m,m+\tfrac12),\\ 
0, & x\in [m+\tfrac12,m+1),
\end{cases}
\qquad
G_{-1}^{g,g}(x)=
\begin{cases}
0, & x\in [m,m+\tfrac12),\\ 
1, & x\in [m+\tfrac12,m+1),
\end{cases}
\end{equation*}
for each $m\in\mathbb Z$.
So the equation $S_{g,g}\widetilde{g}=g$ becomes:
\begin{equation*}
    \begin{cases}
    3\widetilde{g}(x)+\frac32\,\widetilde{g}\!\left(x+\frac32\right)=g(x), &  \text{ on each left half-interval } [m,m+\tfrac12), \\
    3\widetilde{g}(x)+\frac32\,\widetilde{g}\!\left(x-\frac32\right)=g(x) & \text{ on each right half-interval } [m+\tfrac12,m+1)
    \end{cases}
\end{equation*}
Therefore, one obtains that  the canonical dual is the piecewise constant function
\begin{equation*}
\widetilde{g}(x)=
\begin{cases}
-\frac{2}{9}, & x\in[-2,-\tfrac32)\cup[\tfrac32,2),\\ 
\frac{2}{9}, & x\in[-1,-\tfrac12)\cup[\tfrac12,1),\\ 
\frac{4}{9}, & x\in[-\tfrac12,\tfrac12),\\ 
0, & \text{otherwise}.
\end{cases}
\end{equation*}
up to values at the endpoints.
In particular, $\supp(\widetilde{g})=[-2,2]$,
so the canonical dual is compactly supported, even though we are outside the painless regime.
\end{example}

\bibliographystyle{siam}
\bibliography{references}
\end{document}